[1994/12/01]

\documentclass{amsart}
\usepackage{amssymb}
\usepackage{amsmath,amssymb,amsfonts,amsthm,graphics,
latexsym, amscd, amsfonts, epsfig, eepic,epic}
\usepackage{mathrsfs}
\usepackage{algorithm}
\usepackage{algorithmic}
\usepackage{color}

\usepackage{eucal}
\usepackage{eufrak}
\usepackage[all]{xypic}
\usepackage{xspace}
\usepackage{enumerate}

\newtheorem{thm}{Theorem}[section]
\newtheorem{lem}[thm]{Lemma}
\newtheorem{prop}[thm]{Proposition}

\newtheorem{conj}[thm]{Conjecture/Question}

\theoremstyle{definition}
\newtheorem{defn}{Definition}[section]

\theoremstyle{remark}
\newtheorem{rem}{Remark}[section]

\makeatletter \@addtoreset{equation}{section}

\newcommand{\thmref}[1]{Theorem~\ref{#1}}

\newcommand{\lemref}[1]{Lemma~\ref{#1}}

\newcommand{\propref}[1]{Proposition~\ref{#1}}

\def\a{\alpha}
\def\b{\beta}

\def\i{\sqrt{-1}}

\def\l{\lambda}

\def\p{\partial}
\def\vphi{\varphi}

\def\cH{{\cal H}}

\def\tr{\rm tr}

\def\cH{{\mathcal H}}

\def\RR{{\mathbb R}}

\def\and{\quad{\rm and}\quad}

\let\lra=\longrightarrow

\def\mapright{\xrightarrow}
\def\mapright\#1{\,\smash{\mathop{\lra}\limits^{\#1}}\,}

\def\om{\omega}
\def\Om{\Omega}

\def\tri{\triangle}
\def\ul{\underline}

\newcommand{\as}{\underline{S}}

\numberwithin{equation}{section}

\makeatletter

\newcommand{\Rmnum}[1]{\expandafter\@slowromancap\romannumeral #1@}

\title[$I$-properness of Mabuchi's $K$-energy]{$I$-properness of Mabuchi's $K$-energy}

\author[Kai Zheng]{Kai Zheng}
  \address{Institut f\"ur Differentialgeometrie
Gottfried Wilhelm Leibniz Universit\"at Hannover
Welfengarten 1 (Hauptgeb\"aude)
30167 Hannover}
  \email{zheng@math.uni-hannover.de}
\keywords{}
 
\subjclass[2000]{Primary 53C44; Secondary 58E30,  58B20}
\begin{document}
\maketitle
\begin{abstract}
Over the space of K\"ahler metrics associated to a fixed K\"ahler class, we first prove the lower bound of the energy functional $\tilde E^\b$ \eqref{E}, then we provide the criterions of the geodesics rays to detect the lower bound of $\tilde {\mathfrak J}^\b$-functional \eqref{Jbeta}. They are used to obtain the properness of Mabuchi's $K$-energy. The criterions are examined under \eqref{jflowcondition1} by showing the convergence of the negative gradient flow of $\tilde {\mathfrak J}^\b$-functional.
\end{abstract}
\tableofcontents
\section{Introduction}
Let $M$ be a compact K\"ahler manifold and $\Om$ be an arbitrary K\"ahler class. We choose a K\"ahler metric $\om$ in $\Om$ and denote the space of K\"ahler potentials associated to $\Om$ by
$$\cH_\Om=\{\varphi\in C^{\infty}(M, \RR)\;|\;\om_\vphi=\om+\i\p\bar\p \varphi>0\}.$$
Mabuchi's $K$-energy \cite{MR867064} has the explicit formula (cf.
\cite{MR1772078}  \cite{MR1787650}) for any $\vphi\in \cH_\Om$,
\begin{align}\label{K energy}
\nu_\om(\vphi) &=E_\om(\vphi)+\as \cdot D_\om(\vphi)+j_\om(Ric(\om),\vphi).
\end{align}
In which,
\begin{align*}
E_\om(\vphi)&=\int_M \log\frac{\om^n_\vphi}{\om^n}\,\om^n_\vphi ,\\
D_\om(\vphi)&=\frac{1}{V}\int_M\vphi\om^n-J_\om(\vphi),\\
J_\om(\vphi)&=\frac{\i}{V}\sum_{i=0}^{n-1}\frac{i+1}{n+1}
\int_{M}\p\vphi\wedge\bar\p\vphi\wedge\om^{i}\wedge\om^{n-1-i}_{\vphi}.
\end{align*}
and
\begin{align*}
&j_\om(Ric(\om),\vphi)\\
&=\frac{-1}{V}\sum_{i=0}^{n-1}\frac{n!}{(i+1)!(n-i-1)!}\int_{M}\vphi\cdot
Ric(\om)\wedge\om^{n-1-i}\wedge(\i\p\bar\p\vphi)^{i}\nonumber.
\end{align*}
We also recall Aubin's $I$-function,
\begin{align*}
I_\om(\vphi)&=\frac{1}{V}\int_{M}\vphi(\om^{n}-\om_{\vphi}^{n})
=\frac{\i}{V}\sum_{i=0}^{n-1}\int_{M}\p\varphi\wedge\bar\p\vphi\wedge\om^{i}\wedge\om_\vphi^{n-1-i}.
\end{align*}

The properness of the $K$-energy $\nu_\om(\vphi)$ is a kind of "coercive" condition in the variational theory. It was introduced in Tian \cite{MR1471884}, which states that there is a nonnegative, non-decreasing function $\rho(t)$ satisfying $\lim_{t\rightarrow\infty}\rho(t)=\infty$ such that $\nu_\om(\vphi)\geq \rho(I_\om(\vphi))$ for all $\vphi\in\mathcal H_\Om$. It is conjectured to be equivalent to the existence of the constant scalar curvature K\"ahler (cscK) metrics (Conjecture 7.12 in Tian \cite{MR1787650}).

When $\Om=- C_1(M)$ or $C_1(M)=0$, the function $\rho$ is proved to be linear in Tian \cite{MR1787650}, Theorem 7.13, i.e. there are two positive constants $A$ and $B$ such that for all $\vphi\in \mathcal H_\Om$,
\begin{align}\label{Iproper}
\nu_\om(\vphi) \geq A I_\om(\vphi)- B.
\end{align}
In order to destine different notions of properness, in our paper, we say the $K$-energy is $I$-proper, if \eqref{Iproper} holds.

When $\Om=C_1(M)>0$ and there is no holomorphic vector field on a K\"ahler-Einstein manifold $M$, Phong-Song-Sturm-Weinkove \cite{MR2427008} proved that Ding functional $F_\om(\vphi)$ (defined in Ding \cite{MR967024}) satisfies $$F_\om(\vphi) \geq A I_\om(\vphi)- B.$$ This inequality is a generalisation of the Moser-Trudinger inequalities on the sphere \cite{MR677001}\cite{MR0301504}\cite{MR0216286}. The $I$-properness of Ding functional also implies \eqref{Iproper} by using  the identity between $\nu_\om(\vphi)$ and $F_\om(\vphi)$ in Ding-Tian \cite{dingtianf}, we include the proof in \lemref{Ipropertokproper} for readers' convenience.

There are different notions of properness. In \cite{MR2471594}, Chen defined another properness of the $K$-energy
regarding to the entropy $E_\om(\vphi)$. The equivalent relation between the $I$-properness and the $E$-properness is discussed in \cite{MR3229802}. Chen also suggest another properness which means that
  the $K$-energy bounds the geodesic distance function. He furthermore
   conjectured that $d$-properness should be a necessary condition of
    the existence of the cscK or the general extremal K\"ahler metrics (see Conjecture/Question 2
    in \cite{MR1772078} and Conjecture/Question 6.1 in
    \cite{Chen:2008nr}).\\

Let $\chi$ be a closed $(1,1)$-form.
The $\mathfrak{J}$-functional is defined to be the last two terms of the $K$-energy with $Ric(\om)$ replaced by $\chi$,
\begin{align*}
\mathfrak{J}_{\om,\chi}(\vphi)&=\underline S \cdot D_\om(\vphi)+  j_\om(\chi,\vphi).
\end{align*}
We introduce a new parameter $\b$ within a range
$$0\leq\b<\frac{n+1}{n}\a.$$
We then define a new functional to be
\begin{align}\label{Jbeta}
\tilde {\mathfrak J}^\b_{\om,\chi}(\vphi)=\mathfrak J_{\om,\chi}(\vphi)+\b J_\om(\vphi).
\end{align}

Now we return back to the formula of the $K$-energy. With the notations above it is split into
\begin{align}\label{decomposek}
\nu_\om(\vphi)=E_\om(\vphi)-\b J_\om(\vphi)+\tilde {\mathfrak J}^\b_{\om,Ric(\om)}(\vphi).
\end{align}
The lower bound of $E_\om(\vphi)$ is $\a I_{\om}(\varphi) -C$ in \lemref{tianinequality}.
Inserting it into the $K$-energy, we arrive at the lower bound
\begin{align*}
\nu_\om(\vphi)\geq \a I_{\om}(\varphi) -C -\b J_\om(\vphi)+\inf_{\vphi\in\mathcal H_\Om} \tilde{\mathfrak J}^\b_{\om,Ric(\om)}(\vphi).
\end{align*}
Note that $I$-functional is equivalent to the $J$-functional, $$\frac{1}{n+1}I_{\om}(\varphi)\leq J_{\om}(\varphi)\leq\frac{n}{n+1}I_{\om}(\varphi),$$ then we have
\begin{align}\label{iproper}
\nu_\om(\vphi)\geq (\a-\frac{n\b}{n+1}) I_{\om}(\varphi) -C +\inf_{\vphi\in\mathcal H_\Om} \tilde{\mathfrak J}^\b_{\om,Ric(\om)}(\vphi).
\end{align}
From this inequality, we observe that in order to prove the $I$-properness of the $K$-energy, it suffices to obtain the lower bound of the functional $\tilde{\mathfrak J}^\b_{\om,Ric(\om)}$.\\

The critical points of $\tilde{\mathfrak J}^\b_{\om,\chi}$ satisfy a new fully nonlinear equation in $\cH_\Om$,
\begin{align}\label{lma}
n\cdot\chi\wedge\om^{n-1}_\vphi=c_\b \cdot \om^{n}_\vphi+\frac{\b}{V}\om^n.
\end{align}
The constant $c$ is a topological constant determined by $$c_\beta=n\frac{[\chi]\cdot\Om^{n-1}}{\Om^{n}}-\frac{\b}{V}.$$
We call such $\om_\vphi$ a $\tilde{\mathfrak J}^\b$-metric.
We say that $\chi$ is semi-definite
$$\text{ if it is negative semi-definite or positive semi-definite.}$$ In these degenerate situation, \eqref{lma} might have more than one solution. We first prove the lower bound the $\tilde{\mathfrak{J}}^\b$-functional, when there is a $\tilde{\mathfrak{J}}^\b$-metric in $\Om$.
\begin{thm}\label{c1 nega j low bdd}
Assume that $\chi$ is negative semi-definite (positive semi-definite) and there is a $\tilde{\mathfrak{J}}^\b$-metric in $\Om$, then all $\tilde{\mathfrak{J}}^\b$-metrics has the same critical value and $\tilde{\mathfrak{J}}^\b$ has lower (resp. upper) bound.
\end{thm}

There is another functional $\tilde E^\b$ which is defined to be the square norm of the derivative of $\tilde{\mathfrak{J}}^\b$,
\begin{align}\label{E}
\tilde E^\b(\vphi)=\frac{1}{V}\int_M(c_\b-\tr _{\om_\vphi} \chi+\frac{\b}{V}\frac{\om^n}{\om^n_\vphi})^2\om^n_\vphi.
\end{align}
The $\tilde {\mathfrak J}^\b$-function and the $\tilde E^\b$-functional play the roles as the $K$-energy and the Calabi energy in the study of extremal K\"ahler metrics.
We next prove the lower bound of $\tilde E^\b$.

When $\chi$ is semi-definite,  according to the 2nd variation formula of $\tilde{\mathfrak{J}}^\b$ in \eqref{2nd J functional}, it is convex or concave along a $C^{1,1}$ geodesic ray $\rho(t)$. Thus the limit of its first derivative along $\rho(t)$ exists\begin{align}\label{euro}
\mathfrak{F}^\b(\rho)=\lim_{t\rightarrow\infty}\frac{1}{V}\int_{M}\frac{\p \rho}{\p t}(c_\b-\tr _{\om_\vphi} \chi+\frac{\b}{V}\frac{\om^n}{\om^n_\vphi}) \om_\vphi^n.
\end{align}

We require the following notions of the geodesic ray in the space of K\"ahler potentials.
\begin{defn}
We say a $C^{1,1}$ geodesic ray is
\begin{itemize}
  \item stable (semi-stable) if $\mathfrak{F}^\b>0$ ($\mathfrak{F}^\b\geq0$);
  \item destabilising (semi-destabilising) if $\mathfrak{F}^\b<0$ ($\mathfrak{F}^\b\leq0$);
  \item effective if $\limsup_{t\rightarrow\infty} \tilde E^\b(\rho(t))\cdot \frac{1}{t^2}=0.$
\end{itemize}
\end{defn}

\begin{thm}\label{lower bound of E}
Assume that $\chi$ is negative semi-definite.
The following inequality holds.
\begin{align}
\inf_{\om\in \Om}\sqrt{\tilde E^\b} \geq \sup_{\rho}( -\mathfrak{F}^\b ).
\end{align}
The supreme is taking over all $C^{1,1}$, effective, semi-destabilising geodesic $\rho$.
\end{thm}

We remark that when $\b=0$ and $\chi$ and $\om$ are both algebraic, the lower bound of $\tilde E^0$ was proved in Lejmi and Sz\'{e}kelyhidi \cite{Lejmi:2013qf} in algebraic setting.\\

We then prove the lower bound of $\tilde {\mathfrak J}^\b$ without the existence of $\tilde {\mathfrak J}^\b$-metric.
\begin{thm}\label{lowerboundEJ}
Suppose that $\chi$ is negative semi-definite.
Assume that $\tilde{\mathfrak J}^\b$ is bounded from below along a $C^{1,1}$ semi-destabilising geodesic ray and the infimum of the energy $\tilde E^\b$ is zero along this ray. Then $\tilde{\mathfrak J}^\b$ is uniformly bounded from below in the entire K\"ahler class $\Om$.
\end{thm}

The tool we use here to obtain these lower bounds is based on Chen \cite{MR2471594}\cite{Chen:2008nr}. The proof relies on the existence of the geodesic rays and the nonpositive curvature property of the infinite dimensional space $\mathcal H_\Om$. In general, it is difficult to examine the lower bound of functionals in an infinite dimensional space, however, \thmref{lowerboundEJ} provides a method to examine it along only one geodesic ray. \\

Furthermore, we apply \thmref{lowerboundEJ} to the $K$-energy. When $C_1(M)<0$, according to Aubin-Yau's solution of the Calabi conjecture \cite{MR480350}\cite{MR1636569}, there exists a unique K\"ahler metric $\om_0$ such that $Ric(\om_0)$ represents the first Chern class. So let $$\chi=Ric(\om_0)$$ could be chosen to be $<0$. We obtain the following criterion of the $I$-properness of the $K$-energy.

\begin{thm}\label{proper}
Assume that there is a $C^{1,1}$ semi-destabilising geodesic ray $\rho(t)$ such that along $\rho(t)$
\begin{enumerate}
\item $\tilde{\mathfrak J}^\b$ is bounded from below,
\item the infimum of the energy $\tilde E^\b$ is zero.
\end{enumerate} Then the $K$-energy is $I$-proper.
\end{thm}

When $\Om$ admits a $\tilde{\mathfrak J}^\b$-metric $\vphi$, the trivial geodesic ray $\rho(t)=\vphi, \forall t\geq 0$ provides such geodesic ray required in this theorem, since $\mathfrak{F}^\b=0$, the first condition follows from \thmref{c1 nega j low bdd} and the second one follows from \thmref{lower bound of E}.\\

One way to obtain the critical metric of $\mathfrak{J}$-functional is its negative gradient flow. It was introduced in Chen \cite{MR1772078} and also in Donaldson \cite{MR1701920} from moment map picture.
Theorem 1.1 in Song-Weinkove \cite{MR2368374} showed that under the following condition of a K\"ahler class $\Om$, that is, if there is a K\"ahler metric $\om\in \Om$ such that
$-\chi>0$ and
$(-c_0\cdot\om+(n-1)\chi)\wedge\om^{n-2}>0$,
the negative gradient flow of $\mathfrak{J}$-functional converges. Thus $I$-properness \eqref{Iproper} holds when $\chi=Ric(\om_0)\in C_1(M)<0$ and $(-c_0\cdot\om+(n-1)Ric(\om_0))\wedge\om^{n-2}>0$. We extend their theorem to the negative gradient flow of $\tilde{\mathfrak J}^\b$-functional
\begin{align}\label{jflow}
\frac{\p \vphi}{\p t}&=-c_\beta+\frac{n\chi\wedge\om_\vphi^{n-1}}{\om_\vphi^n}
 -\frac{\b}{V}\frac{\om^n}{\om^n_\vphi}.
\end{align} and prove its convergence in Proposition \ref{conjflow} under the condition,
\begin{align*}
-\chi>0 \text{ and }(-c_\b\cdot\om+(n-1)\chi)\wedge\om^{n-2}>0.
\end{align*}
The extra term involving $\b$ on the flow equation brings us trouble when we apply the second order estimate.
In order to overcome this problem, we calculate a differential inequality by using the linear elliptic operator $L$ defined in \eqref{Loperator} and apply the maximum principal.

We remark that \eqref{lma} and its flow have been generalised in different directions \cite{MR3091809}\cite{MR3219504}\cite{MR2794631}\cite{Li:2013yg}... which is far from a complete list.

Thus we verify the criterion in \thmref{proper}.

\begin{thm}\label{constgeodesicray}
Assume that there is a $\om\in\Om$ such that
\begin{align}\label{jflowcondition1}
(-c_\b\cdot\om+(n-1)Ric(\om_0))\wedge\om^{n-2}>0.
\end{align} Then from any K\"ahler potential $\vphi\in\cH_\Om$, there exists a $C^{1,1}$ semi-destabilising geodesic ray satisfying $(1)$ and $(2)$. Thus the $K$-energy is $I$-proper in $\Om$.
\end{thm}

Paralleling to Donaldson's conjecture of existence of the cscK metrics (Conjecture/Question 12 in \cite{MR1736211}), we propose a notion called geodesic stability w.r.t to the $\tilde{\mathfrak{J}}^\b$-functional (see Definition \ref{geosta}).
We at last link the existence of $\tilde{\mathfrak{J}}^\b$-metric to this geodesic stability.
\begin{thm}\label{geodesicstability}Suppose that $\chi$ is negative semi-definite.
Assume that $\Om$ contains a $\tilde{\mathfrak{J}}^\b$-metric $\vphi$, then $\Om$ is \emph{geodesic semi-stable} at $\vphi$ and moreover, it is weak geodesic semi-stable.
\end{thm}

The criterion \eqref{criterion} means that along the geodesic ray, the first derivative of the $\tilde{\mathfrak{J}}^\b$-functional is strictly increase.
The question \ref{conj} suggests that there is no such geodesic ray satisfying \eqref{criterion} implies the existence of $\tilde{\mathfrak{J}}^\b$-metric. Then from \thmref{c1 nega j low bdd} and \eqref{iproper}, the $K$-energy is $I$-proper. So according to Tian's conjecture (Conjecture 7.12 in \cite{MR1787650}), there exists cscK metrics. In this sense, the question \ref{conj} probably provides another possible point of view of Donaldson's conjecture (Conjecture/Question 12 in \cite{MR1736211}).

We further remark that with these theorems, it would be more interesting to find the examples of K\"ahler class where the $\tilde{\mathfrak{J}}^\b$-functional has lower bound but the $\tilde{\mathfrak{J}}^\b$-metric does not exist.






\section{Variational structure of $\tilde{ \mathfrak J}$ and $\tilde E$}\label{Lagrangian functionals}

Recall our definition for any $\vphi\in \cH_\Om$,
\begin{align*}
\tilde {\mathfrak J}^\b_{\om,\chi}(\vphi)=\mathfrak J_{\om,\chi}(\vphi)+\b \cdot J_\om(\vphi).
\end{align*}
Let $\vphi(t)$ be a smooth family of K\"ahler potentials with $\vphi(0)=\vphi$. We denote $$\delta=\frac{d}{dt}\vert_{t=0}\text{ and } \dot\vphi=\delta\vphi(t).$$
\begin{lem}
The 1st variation of $\tilde{\mathfrak J}$-functional is
\begin{align*}
\delta\tilde {\mathfrak J}^\b(\dot\vphi)=\frac{1}{V}\int_{M}\dot\vphi[c_\b \cdot \om_\vphi^n-n\cdot
\chi\wedge\om^{n-1}_{\vphi}+ \frac{\b}{V}\om^n].
\end{align*}
\end{lem}
\begin{proof}
We compute
\begin{align*}
\delta\tilde {\mathfrak J}^\b(\dot\vphi)
&=\frac{1}{V}\int_{M}\dot\vphi(c_0 \cdot \om_\vphi^n-n\cdot
\chi\wedge\om^{n-1}_{\vphi})+ \frac{\b}{V}\int_{M} \dot\vphi(\om^n-\om^{n}_{\vphi})\\
&=\frac{1}{V}\int_{M}\dot\vphi[(c_0-\frac{\b}{V}) \cdot \om_\vphi^n-n\cdot
\chi\wedge\om^{n-1}_{\vphi}+ \frac{\b}{V}\om^n].
\end{align*}
\end{proof}
\begin{lem}
The 2nd variation of $\tilde{\mathfrak J}$-functional is
\begin{align}\label{2nd J functional}
\delta^2\tilde {\mathfrak J}^\b(\dot\vphi,\dot\vphi)=\frac{1}{V}\int_{M}(\ddot\vphi-\vert\p\dot\vphi\vert^2)(c_\b-\tr _{\om_\vphi} \chi)\om_\vphi^n
-\frac{1}{V}\int_M \chi_{i\bar j}\dot\vphi^i\dot\vphi^{\bar j}\om_\vphi^n.
\end{align}
\end{lem}
\begin{proof}
We compute directly,
\begin{align*}
\delta^2\tilde {\mathfrak J}^\b(\dot\vphi,\dot\vphi)
&=\frac{d}{dt}\frac{1}{V}\int_{M}\dot\vphi[(c_\b -g_\vphi^{i\bar j}\chi_{i\bar j}) \om_\vphi^n +\frac{\b}{V}\om^n]\\
&=\frac{1}{V}\int_{M}\ddot\vphi(c_\b -g_\vphi^{i\bar j}\chi_{i\bar j}) \om_\vphi^n
+\frac{1}{V}\int_{M}\dot\vphi\dot\vphi^{i\bar j}\chi_{i\bar j} \om_\vphi^n\\
&+\frac{1}{V}\int_{M}\dot\vphi(c_\b-g_\vphi^{i\bar j}\chi_{i\bar j})\tri_\vphi\dot\vphi \om_\vphi^n.
\end{align*}
The second term becomes
\begin{align*}
&\frac{1}{V}\int_{M}\dot\vphi\dot\vphi^{i\bar j}\chi_{i\bar j} \om_\vphi^n\\
&=-\frac{1}{V}\int_{M}\dot\vphi^{\bar j}\dot\vphi^{i}\chi_{i\bar j} \om_\vphi^n
-\frac{1}{V}\int_{M}\dot\vphi\dot\vphi^{i}(\chi_{i\bar j} )^{\bar j}\om_\vphi^n\\
&=-\frac{1}{V}\int_{M}\dot\vphi^{\bar j}\dot\vphi^{i}\chi_{i\bar j} \om_\vphi^n
-\frac{1}{V}\int_{M}\dot\vphi\dot\vphi^{i}(\chi_{j\bar j} )_i\om_\vphi^n.
\end{align*}
The third term is
\begin{align*}
&\frac{1}{V}\int_{M}\dot\vphi(c_\b-g_\vphi^{i\bar j}\chi_{i\bar j})\tri_\vphi\dot\vphi \om_\vphi^n\\
&=-\frac{1}{V}\int_{M}(c_\b-g_\vphi^{i\bar j}\chi_{i\bar j})\vert\p\dot\vphi\vert^2 \om_\vphi^n
+\frac{1}{V}\int_{M}\dot\vphi(g_\vphi^{i\bar j}\chi_{i\bar j})_{\bar l}\dot\vphi_k g_\vphi^{k\bar l} \om_\vphi^n.
\end{align*}
Then the lemma follows from adding them together.
\end{proof}

Therefore, when $\chi$ is strictly negative (positive), the $\tilde {\mathfrak J}^\b$-metric is local minimum (maximum).

\begin{prop}\label{uniqueness}
When $\chi$ is strictly negative or strictly positive, the $\tilde {\mathfrak J}^\b$-metric is unique up to a constant.
\end{prop}
\begin{proof}
Assume $\vphi_1$ and $\vphi_2$ are two $\tilde {\mathfrak J}^\b$-metrics. Then connecting them by the $C^{1,1}$ geodesic. Since all the computation above is well-defined along the $C^{1,1}$ geodesics, \eqref{2nd J functional} implies that
$$\delta^2\tilde {\mathfrak J}(\dot\vphi,\dot\vphi)\
=-\frac{1}{V}\int_M \chi_{i\bar j}(\om) \dot\vphi^i\dot\vphi^{\bar j}\om_\vphi^n.$$ Then integrating from $0$ to $1$, we have $$\frac{1}{V}\int_0^1\int_M \chi_{i\bar j}(\om) \dot\vphi^i\dot\vphi^{\bar j}\om_\vphi^ndt=\delta\tilde {\mathfrak J}(1)-\delta\tilde {\mathfrak J}(0)=0.$$ Hence, $\dot\vphi$ is constant and $\vphi_1$ and $\vphi_2$ differ by a constant.
\end{proof}

We use the notion
$$\tilde H = \tr _{\om_\vphi} \chi-c_\b-\frac{\b}{V}\frac{\om^n}{\om^n_\vphi}.$$
The $\tilde{\mathfrak{J}}^\b$-metric is a K\"ahler metric satisfying $$\tilde H=0.$$
We define the energy $\tilde E^\b$ as
\begin{align}\label{tildeE}
\tilde E^\b(\vphi)=\frac{1}{V}\int_M(\tr _{\om_\vphi} \chi-c_\b-\frac{\b}{V}\frac{\om^n}{\om^n_\vphi})^2\om^n_\vphi.
\end{align}
Then we have
$$\delta\tilde H(\dot\vphi)= -\dot\vphi^{i\bar j}\chi_{i\bar j}+\frac{\b}{V}\tri_\vphi\dot\vphi\frac{\om^n}{\om^n_\vphi}.$$
\begin{lem}\label{d E}The 1st derivative of the modified energy $\tilde E$ is
\begin{align}\label{d E}
\delta \tilde E^\b(\dot\vphi)=\frac{2}{V}\int_M\tilde H^{\bar j}\dot\vphi^{i}\chi_{i\bar j}\om^n_\vphi
 -\frac{2\b}{V^2}\int_M\tilde H_i \dot\vphi^i  \om^n.
\end{align}
\end{lem}
\begin{proof}
We calculate that
\begin{align*}
\delta\tilde E^\b(\dot\vphi)=\frac{2}{V}\int_M\tilde H(-\dot\vphi^{i\bar j}\chi_{i\bar j}+\frac{\b}{V}\tri_\vphi\dot\vphi\frac{\om^n}{\om^n_\vphi})\om^n_\vphi
+\frac{1}{V}\int_M\tilde H^2\tri_\vphi\dot\vphi\om^n_\vphi.
\end{align*}
The first term is
\begin{align*}
&\frac{2}{V}\int_M\tilde H^{\bar j}\dot\vphi^{i}\chi_{i\bar j}\om^n_\vphi
+\frac{2}{V}\int_M\tilde H\dot\vphi^{i}(\chi_{i\bar j})^{\bar j}\om^n_\vphi\\
&=\frac{2}{V}\int_M\tilde H^{\bar j}\dot\vphi^{i}\chi_{i\bar j}\om^n_\vphi
+\frac{2}{V}\int_M\tilde H\dot\vphi^{i}(\tilde H+\frac{\b}{V}\frac{\om^n}{\om^n_\vphi})_i\om^n_\vphi.
\end{align*}
While, the second term is
\begin{align*}
&\frac{2}{V}\int_M\tilde H \frac{\b}{V}\tri_\vphi\dot\vphi \frac{\om^n}{\om^n_\vphi} \om^n_\vphi\\
&=-\frac{2}{V}\int_M\tilde H_i \frac{\b}{V}\dot\vphi^i \frac{\om^n}{\om^n_\vphi} \om^n_\vphi-
\frac{2}{V}\int_M\tilde H \frac{\b}{V}\dot\vphi^i (\frac{\om^n}{\om^n_\vphi})_i \om^n_\vphi
\end{align*}
and the third term is
\begin{align*}
-\frac{2}{V}\int_M\tilde H g_\vphi^{i\bar j}\tilde H_{\bar j}\dot\vphi_i\om^n_\vphi
\end{align*}
which cancels the second component in the first term.
\end{proof}

The critical points of $\tilde E$ satisfy that
\begin{align*}
[\tilde H^{\bar j}\chi_{i\bar j}\om^n_\vphi
 -\frac{\b}{V}\tilde H_i \frac{\om^n}{ \om_\vphi^n}]^i=0.
\end{align*}
\begin{lem}
The 2nd derivative of the modified energy $\tilde E^\b$ is
\begin{align}\label{dd E h0}
\delta^2 \tilde  E^\b(u,v)
&=\frac{2}{V}\int_M(v^{p\bar q}\chi_{p\bar q})( u^{i\bar j}\chi_{i\bar j} ) \om^n_\vphi+\frac{2\b}{V^2}\int_Mg_\vphi^{k \bar j} (\tri_\vphi v\frac{\om^n}{\om^n_\vphi})_k g_\vphi^{i\bar l}u_{\bar l}\chi_{i\bar j}\om^n_\vphi\\
&
-\frac{2\b}{V^2}\int_Mg_\vphi^{i\bar j} (-v^{p\bar q}\chi_{p\bar q}+\frac{\b}{V}\tri_\vphi v\frac{\om^n}{\om^n_\vphi})_i u_{\bar j}  \om^n\nonumber.
\end{align}
\end{lem}
\begin{proof}
In the local coordinante, \eqref{d E} is written as
\begin{align*}
 \delta\tilde E^\b(u)=\frac{2}{V}\int_Mg_\vphi^{k \bar j} \tilde H_k g_\vphi^{i\bar l}u_{\bar l}\chi_{i\bar j}\om^n_\vphi-\frac{2\b}{V^2}\int_Mg_\vphi^{i\bar j}\tilde H_i u_{\bar j}  \om^n,
\end{align*}
we obtain that
\begin{align}\label{dd E}
& \delta^2\tilde  E^\b(u,v)\\
&=-\frac{2}{V}\int_Mv^{k \bar j}\tilde H_k g_\vphi^{i\bar l}u_{\bar l}\chi_{i\bar j}\om^n_\vphi
+\frac{2}{V}\int_Mg_\vphi^{k \bar j} (-v^{p\bar q}\chi_{p\bar q}+\frac{\b}{V}\tri_\vphi v\frac{\om^n}{\om^n_\vphi})_k g_\vphi^{i\bar l}u_{\bar l}\chi_{i\bar j}\om^n_\vphi\nonumber\\
&
-\frac{2}{V}\int_Mg_\vphi^{k \bar j}\tilde H_k v^{i\bar l}u_{\bar l}\chi_{i\bar j}\om^n_\vphi
+\frac{2}{V}\int_Mg_\vphi^{k \bar j}\tilde H_k g_\vphi^{i\bar l}u_{\bar l}\chi_{i\bar j}\tri_\vphi v \om^n_\vphi\nonumber\\
&+\frac{2\b}{V^2}\int_Mv^{i\bar j}\tilde H_i u_{\bar j}  \om^n
-\frac{2\b}{V^2}\int_Mg_\vphi^{i\bar j} (-v^{p\bar q}\chi_{p\bar q}+\frac{\b}{V}\tri_\vphi v\frac{\om^n}{\om^n_\vphi})_i u_{\bar j}  \om^n\nonumber.
\end{align}
The second term is further reduced to,
\begin{align*}
&-\frac{2}{V}\int_Mg_\vphi^{k \bar j} (v^{p\bar q}\chi_{p\bar q})_k g_\vphi^{i\bar l}u_{\bar l}\chi_{i\bar j}\om^n_\vphi\\
&=-\frac{2}{V}\int_M(v^{p\bar q}\chi_{p\bar q})^{ \bar j}  u^i\chi_{i\bar j}\om^n_\vphi\\
&=\frac{2}{V}\int_M(v^{p\bar q}\chi_{p\bar q})( u^{i\bar j}\chi_{i\bar j} ) \om^n_\vphi
+\frac{2}{V}\int_M(v^{p\bar q}\chi_{p\bar q})  u^i \tilde H_i\om^n_\vphi.
\end{align*}
Thus the lemmas holds by inserting this formula into \eqref{dd E}.
\end{proof}

When $\beta=0$, the variational structure of $\tilde{\mathfrak J}_{\om,\chi}^0$ and $\tilde E^0$ is studied in  Chen \cite{MR2104078}. We denote
\begin{align*}
H=\tr _{\om_\vphi} \chi-c_0.
\end{align*}
The K\"ahler metric is called a $\mathfrak{J}$-metric if it satisfies $H=0$.
From \eqref{d E}, the 1st derivative of $\tilde E^0$-energy is
\begin{align}\label{ddd E}
\delta \tilde E^0(\dot\vphi)=\frac{2}{V}\int_MH^{\bar j}\dot\vphi^{i}\chi_{i\bar j}\om^n_\vphi.
\end{align}
From this formula, the critical metrics satisfy the equation
\begin{align}
[H^{\bar j}\chi_{i\bar j}]^{i}=0.
\end{align}
The critical metrics of the modified energy include the $\mathfrak{J}$-metrics.
\eqref{dd E h0} shows that, at the critical point of $\mathfrak{J}$,
\begin{align*}
\delta^2\tilde E^0(u,v)
=\frac{2}{V}\int_M(v^{p\bar q}\chi_{p\bar q})( u^{i\bar j}\chi_{i\bar j} ) \om^n_\vphi.
\end{align*}
So the $\mathfrak{J}$-metric is local minimiser of $ \tilde E^0$. However, it is not known whether all the critical metrics of the energy $\tilde E^0$ are minimisers.
While, \eqref{ddd E} suggests that when $\chi$ is strictly positive or negative, the critical metrics of the modified energy is the $\mathfrak{J}$-metric.

\section{Geodesics in the space of K\"aher potentials}
We recall the necessary progress of constructing the geodesic ray in this section for the next several sections. the existence of the $C^{1,1}$ geodesic segment is proved in Chen \cite{MR1863016}.
In Calamai-Zheng \cite{Calamai:2012uq}, we improve the following existence of the geodesic segment with slightly weaker boundary geometric conditions. Now we specify the geometric conditions on the boundary metrics.
\begin{defn}\label{spaceHC}
We label as $\mathcal H_C \subset \mathcal H_\Om$ one of the following spaces;
\begin{align*}
\mathfrak{I_1}&=\{\vphi\in \mathcal H_\Om\text{ {such that} }  \sup Ric(\om_\vphi)\leq C
\} ;\\
\mathfrak{I_2}&=\{\vphi\in\mathcal H_\Om\text{ {such that} } \inf Ric(\om_\vphi)\geq C
\} .
\end{align*}
\end{defn}
\begin{thm}\label{geo seg}(Calamai-Zheng \cite{Calamai:2012uq})
Any two K\"ahler metrics in $\mathcal H_C$ are connected by a unique $C^{1,1}$ geodesic. More precisely, it is the limit under the $C^{1,1}$-norm by a sequence of $C^{\infty}$ approximate geodesics.
\end{thm}

Due to Calabi-Chen \cite{MR1969662}, $\mathcal H$ has positive semi-definite curvature in the sense of Aleksandrov. Two geodesic ray $\rho_i$ are called paralleling if the geodesic distance between $\rho_1(t)$ and $\rho_2(t)$ is uniformly bounded.
\begin{lem}
Given a geodesic ray $\rho(t)$ in $\mathcal H_C$ and a K\"ahler potential $\vphi_0$ which is not in $\rho(t)$. There is a $C^{1,1}$ geodesic ray starting from $\vphi_0$ and paralleling to $\rho(t)$.
\end{lem}
\begin{proof}
According to \thmref{geo seg} we could connect $\vphi_0$ and $\rho(t)$ by a $C^{1,1}$ geodesic segment $\gamma_t(s)$ which have uniform $C^{1,1}$ norm. Thus after taking limit of the parameter $t$, we obtain a limit geodesic ray in $W^{2,p},\forall p\geq1$ and $C^{1,\a},\forall \a<1$,
$$\gamma(s)=\lim_{t\rightarrow\infty}\gamma_t(s).$$
\end{proof}
\begin{rem}
The condition of $\rho(t)$ could be weakened to be the tamed condition in Chen \cite{MR2471594}. We only require that there is a $\tilde\rho(t)\in \mathcal H_C$ and $\tilde\rho(t)-\rho(t)$ is uniformly bounded.
\end{rem}
\section{A functional inequality of $\tilde {\mathfrak J}^\b$ and $\tilde E^\b$}
We first prove a functional inequality.
\begin{prop}
Let $\vphi_0$ and $\vphi_1$ be two K\"ahler potentials then the following inequality holds.
\begin{align*}
\tilde {\mathfrak J}^\b(\vphi_1)-\tilde {\mathfrak J}^\b(\vphi_0)\leq d(\vphi_0,\vphi_1)\cdot \sqrt{ \tilde E^\b(\vphi_1)  }.
\end{align*}
\end{prop}
\begin{proof}
The functional inequality is proved by direct computation. Let $\rho(t)$ be a $C^{1,1}$ geodesic segment connecting $\vphi_0$ and $\vphi_1$.
\begin{align*}
\tilde {\mathfrak J}^\b(\vphi_1)-\tilde {\mathfrak J}^\b(\vphi_0)
&\leq \int_0^1 d\tilde {\mathfrak J}^\b(\frac{\p\rho}{\p t})_{\vphi_1} dt\\
&\leq \sqrt{\frac{1}{V}\int_M\tilde H^2\om^n_{\vphi_1}}\cdot \sqrt{\int_0^1\int_M (\frac{\p\rho}{\p t})^2\om^n_{\vphi_1}dt }.
\end{align*}
Thus the resulting inequality follows from the H\"older inequality.
\end{proof}


\section{Proof of \thmref{c1 nega j low bdd}}
\begin{proof}
Let $\vphi_1$ be any K\"ahler potential in $\mathcal H_\Om$ and $\vphi_0$ be a $\tilde {\mathfrak J}^\b$-metric. Connecting $\vphi_1$ and $\vphi_0$ by a $C^{1,1}$ geodesic segment $\gamma(t)$ and computing the expansion formula along $\gamma(t)$
\begin{align*}
\tilde {\mathfrak J}^\b(1)-\tilde {\mathfrak J}^\b(0)
&=\int_0^1 \frac{\p \tilde {\mathfrak J}^\b}{\p t} dt\\
&=\int_0^1 \frac{\p \tilde {\mathfrak J}^\b}{\p t}(t)- \frac{\p \tilde {\mathfrak J}^\b}{\p t}(0)dt\\
&=\int_0^1\int_0^t \frac{\p^2 \tilde {\mathfrak J}^\b}{\p t^2} dsdt.
\end{align*}
In the second identify we use the assumption that $\vphi_0$ is a $\tilde {\mathfrak J}^\b$-matric, so $$\frac{\p \tilde {\mathfrak J}^\b}{\p t}(0)=0.$$
Applying the 2nd formula of the $\tilde {\mathfrak J}^\b$, \lemref{2nd J functional}, we see that $$(\tilde {\mathfrak J}^\b)''\geq 0$$ along $\gamma(t)$. As a result, we obtain that $$\tilde {\mathfrak J}^\b(1)\geq\tilde {\mathfrak J}^\b(0).$$
Furthermore, assume that $\vphi_1$ is another $\tilde {\mathfrak J}^\b$-metric when the solution is not unique, then we have $$\tilde {\mathfrak J}^\b(1)\geq\tilde {\mathfrak J}^\b(0).$$ Switching the positions of $\vphi_0$ and $\vphi_1$, we see that all $\tilde {\mathfrak J}^\b$-metrics has the same critical value of $\tilde {\mathfrak J}^\b$.
\end{proof}


\section{Proof of \thmref{lower bound of E}}

\begin{proof}
Let $\rho(t)$ be a geodesic ray parameterized by the arc length and satisfy the assumption in the theorem.
Let $\vphi_0$ be a K\"ahler potential outside $\rho(t)$ and connecting $\vphi_0$ and $\rho(t)$ by a $C^{1,1}$ geodesic $\gamma_t(s)$ which is also parameterized by the arc length. Let $\theta$ be the angle expanding by $\overrightarrow{\rho(t)\rho(0)}$ and $\overrightarrow{\rho(t)\vphi(0)}$.

Since $\mathcal H_\Om$ is nonpositive curve, we obtain
$$d(\vphi_0,\rho(0))\geq d$$
by comparing the cosine formulae in the Euclidean space $$d^2=d^2(\vphi_0,\rho(t))+d^2(\rho(0),\rho(t))-2d(\vphi_0,\rho(t))d(\rho(0),\rho(t))\cos\theta.$$ 
Then knowing that $$d(\rho(0),\rho(t))=t,$$ and letting $d_t=d(\vphi_0,\rho(t))$ be the distance between $\vphi_0$ and $\rho(t)$, we have
\begin{align*}
d_0^2&\geq d_t^2+t^2-2d_t \cdot t\cdot \cos\theta\\
&=d_t^2+t^2-2d_t \cdot t +2d_t \cdot t-2d_t \cdot t\cdot\cos\theta\\
&\geq 2d_t \cdot t\cdot (1-\cos\theta).
\end{align*}
Thus the cosine formula implies
\begin{align*}
2(1-\cos\theta)
\leq \frac{d^2_0}{t \cdot d_t}.
\end{align*}
While, the triangle inequality implies that
$$t-d_0\leq d_t\leq t+d_0.$$
When $t$ is sufficient large, we further have 
$$d_0\leq \frac{t}{2}.$$
Thus 
\begin{align}\label{lowerbound tri}
0&\leq 2 (1-(\frac{\p \rho}{\p t},\frac{\p\gamma}{\p s}))_{\rho(t)}\\
&=2(1-\cos\theta)\nonumber\\
&\leq \frac{d^2_0}{t \cdot d_t}\nonumber\\
&\leq \frac{d^2_0}{t \cdot (t-d_0)}\nonumber\\
&\leq \frac{2d^2_0}{t ^2}\nonumber.
\end{align}

Applying the H\"older inequality to
\begin{align*}
d\tilde {\mathfrak J}^\b(\frac{\p\gamma}{\p s})_{\rho(t)}
\leq d\tilde {\mathfrak J}^\b(\frac{\p\gamma}{\p s}-\frac{\p\rho}{\p t})_{\rho(t)}+d\tilde {\mathfrak J}^\b(\frac{\p\rho}{\p t})_{\rho(t)},
\end{align*}
then using \eqref{lowerbound tri}, we obtain
\begin{align}\label{key inequality}
d\tilde {\mathfrak J}^\b(\frac{\p\gamma}{\p s})_{\rho(t)}
&\leq \sqrt{\tilde E^\b(\rho(t))}\sqrt{2-2(\frac{\p\gamma}{\p s},\frac{\p\rho}{\p t})_{\rho(t)}}+d\tilde {\mathfrak J}^\b(\frac{\p\rho}{\p t})_{\rho(t)}\nonumber\\
&\leq \sqrt{\tilde E^\b(\rho(t))}\frac{\sqrt{2}\cdot d_0}{t}+d\tilde {\mathfrak J}^\b(\frac{\p\rho}{\p t})_{\rho(t)}.
\end{align}
Since $\rho(t)$ is effective $$\tilde E^\b(\rho(t))=o(t) t^2,$$ the first term becomes $o(t)$. Then \begin{align}\label{angle 2 geodesic}
d\tilde {\mathfrak J}^\b(\frac{\p\gamma}{\p s})_{\rho(t)}
&\leq o(t)+d\tilde {\mathfrak J}^\b(\frac{\p\rho}{\p t})_{\rho(t)}.
\end{align}

On the other hand, note that $(\tilde {\mathfrak J}^\b)'$ and $(\tilde {\mathfrak J}^\b)''$ are well-defined along $C^{1,1}$ geodeisc. When $\chi$ is negative semi-definite, from \lemref{2nd J functional},
\begin{align*}
(\tilde{\mathfrak{J}}^\b)''(\gamma(s))\geq 0.
\end{align*}
So $$d\tilde {\mathfrak J}^\b(\frac{\p\gamma}{\p s})_{\vphi(0)}
\leq d\tilde {\mathfrak J}^\b(\frac{\p\gamma}{\p s})_{\rho(t)}.$$
Thus combining \eqref{angle 2 geodesic}, we have
\begin{align*}
d\tilde {\mathfrak J}^\b(\frac{\p\gamma}{\p s})_{\vphi(0)}
&\leq o(t)+d\tilde {\mathfrak J}^\b(\frac{\p\rho}{\p t})_{\rho(t)}.
\end{align*}

Inverting this inequality,
\begin{align}\label{angle 2 geodesic1}
-o(t)-d\tilde {\mathfrak J}^\b(\frac{\p\rho}{\p t})_{\rho(t)}
\leq -d\tilde {\mathfrak J}^\b(\frac{\p\gamma}{\p s})_{\vphi(0)}.
\end{align}
The right hand side is controlled by the H\"older inequality again
\begin{align*}
\sqrt{\tilde E^\b(\vphi_0)}\cdot(\int_M(\frac{\p\gamma}{\p s})^2\vert_{s=0}\om^n_{\vphi_0})^\frac{1}{2}=\sqrt{\tilde E^\b(\vphi_0)}.
 \end{align*}
 The inequality follows from choosing the unit arc-length of $\gamma$.
Taking $t\rightarrow\infty$ on both sides of \eqref{angle 2 geodesic1},
\begin{align*}
-\mathfrak{F}^\b(\rho)
&\leq \sqrt{\tilde E^\b(\vphi_0)}.
 \end{align*}
 Thus the theorems follows.
\end{proof}

\section{Proof of \thmref{lowerboundEJ}}
\begin{proof}
Since when $\chi$ is negative semi-definite, $({\tilde {\mathfrak J}}^\b)''\geq 0$ along geodesic ray $\gamma_t(s)$, $\frac{\p {\tilde {\mathfrak J}}^\b}{\p s}$ is non-decreasing. Then letting $\tau(t)$ be the length of the $\gamma_t(s)$, we have
\begin{align*}
\tilde {\mathfrak J}^\b(\rho(t))-\tilde {\mathfrak J}^\b(\vphi_0)
&=\int_0^{\tau(t)}d\tilde {\mathfrak J}^\b(\frac{\p\gamma}{\p s}) ds\\
&\leq \int_0^{\tau(t)}d\tilde {\mathfrak J}^\b(\frac{\p\gamma}{\p s})_{\rho(t)} ds.
\end{align*}

From \eqref{key inequality} in the proof above, we obtain that
\begin{align}\label{constructray}
d\tilde {\mathfrak J}^\b(\frac{\p\gamma}{\p s})_{\rho(t)}
&\leq \sqrt{\tilde E^\b(\rho(t))}\sqrt{2-2(\frac{\p\gamma}{\p s},\frac{\p\rho}{\p t})_{\rho(t)}}+d\tilde {\mathfrak J}^\b(\frac{\p\rho}{\p t})_{\rho(t)}\nonumber\\
&\leq \sqrt{\tilde E^\b(\rho(t))}\frac{\sqrt {2}\cdot d_0}{t}+d\tilde {\mathfrak J}^\b(\frac{\p\rho}{\p t})_{\rho(t)}.
\end{align}
From the assumption that $\rho(t)$ is semi-destabilising, so $$d\tilde {\mathfrak J}^\b(\frac{\p\rho}{\p t})_{\rho(t)}\leq 0.$$
Putting the inequalities above together, we arrive at
\begin{align*}
\tilde {\mathfrak J}^\b(\rho(t))-\tilde {\mathfrak J}^\b(\vphi_0)
\leq \sqrt{\tilde E^\b(\rho(t))}\frac{C\cdot d(\vphi_0,\rho(0))}{t} \tau(t).
\end{align*}
Taking limit of $t$, since $$\tau(t)=O(t)$$ and from assumption in \thmref{lowerboundEJ} along $\rho(t)$, $$\lim_{t\rightarrow\infty}\sqrt{\tilde E^\b(\rho(t))}=0,$$ we have
$$\tilde {\mathfrak J}^\b(\vphi_0)\geq \lim_{t\rightarrow\infty} \tilde {\mathfrak J}^\b(\rho(t)).$$
Thus the theorem follows from the assumption that $\tilde {\mathfrak J}^\b$ is bounded below along $\rho(t)$.
\end{proof}



\section{Geodesic stability}
Inspired from the geodesic conjecture of the extremal metrics in Donaldson \cite{MR1736211}, we proposal a counterpart of $\tilde{\mathfrak{J}}^\b$-metric.
\begin{conj}\label{conj}
The following are equivalent:
\begin{enumerate}
  \item There is no $\tilde{\mathfrak{J}}^\b$-metric in $\mathcal{H}_\Om$.
  \item There is infinite geodesic ray $\vphi(t)$, $t\in[0,\infty)$, in $\mathcal {H}_\Om$ such that 
  \begin{align}\label{criterion}
  \frac{1}{V}\int_{M}\frac{\p \vphi}{\p t}(c_\b-\tr _{\om_\vphi} \chi+\frac{\b}{V}\frac{\om^n}{\om^n_\vphi})\om_\vphi^n>0
  \end{align} for all $t\in[0,\infty)$.
  \item For any point $\vphi\in\mathcal{H}_\Om$, there is a geodesic ray in $(2)$ starting at $\vphi$.
\end{enumerate}
\end{conj}
We need some definitions.
\begin{defn}\label{geosta}
A K\"ahler class is called
\begin{itemize}
  \item \emph{geodesic semi-stable} at a point $\vphi_0$ if every non-trivial $C^{1,1}$ geodesic ray starting from $\vphi_0$ is semi-stable.
  \item \emph{geodesic semi-stable} if every non-trivial $C^{1,1}$ geodesic ray is semi-stable.
  \item \emph{weak geodesic semi-stable} if every non-trivial geodesic ray with uniform $C^{1,1}$ bound is semi-stable.
\end{itemize}
\end{defn}
We say a $C^{1,1}$ geodesic ray is trivial if it is just a point.

\begin{prop}\label{c1 nega geo conv}
Suppose that $\chi$ is negative semi-definite.
We assume that there is a $C^{1,1}$ geodesic ray $\rho(t)$ staying in $\mathcal H_C$ and the $\tilde{\mathfrak{J}}^\b$-functional is non-increasing along $\rho(t)$. If there is a $\tilde{\mathfrak{J}}^\b$-metric, then $\rho(t)$ converges to the $\tilde{\mathfrak{J}}^\b$-metric.
\end{prop}
\begin{proof}
Let $\vphi_0$ be a $\tilde{\mathfrak{J}}^\b$-metric. We first connect $\vphi_0$ and $\rho(t)$ by a $C^{1,1}$ geodesic segment $\gamma_t(s)$, this follows from \thmref{geo seg} since $\rho(t)\in\mathcal H_C$. Moreover, since the $C^{1,1}$ norm is uniform, after taking limit on $t$, we obtain a $C^{1,1}$ geodesic ray $\gamma(s)$ starting at $\vphi_0$. Thus, $\tilde{\mathfrak{J}}^\b$ strongly converges and is well-defined along $\gamma(s)$.

Since the $\tilde{\mathfrak{J}}^\b$ is non-increasing along $\rho(t)$, so $\tilde{\mathfrak{J}}^\b$ has upper bound along $\gamma(s)$. While, \thmref{c1 nega j low bdd} implies that when $\Om$ has a $\tilde{\mathfrak{J}}^\b$-metric, then $\tilde{\mathfrak{J}}^\b$ has lower bound.

Meanwhile, when $\chi$ is negative semi-definite, from \lemref{2nd J functional}, $\tilde{\mathfrak{J}}^\b$ is convex along the geodesic ray $\gamma(s)$. Moreover, $\tilde{\mathfrak{J}}^\b$ obtains its lower bound at $s=0$. So, we claim that $\tilde{\mathfrak{J}}^\b(s)\equiv\min \tilde{\mathfrak{J}}^\b$ along $\gamma(s)$. I.e. $\gamma(s)$ are constituted of $\tilde{\mathfrak{J}}^\b$-metrics.

We prove this claim by the contradiction method.
Since along $\gamma(s)$, the first derivative $(\tilde{\mathfrak{J}}^\b)'$ is non-negative, we assume that $s_0$ is the first finite time such that $(\tilde{\mathfrak{J}}^\b)'$ is strictly positive, otherwise, the claim is proved. Since along $\gamma(s)$, $(\tilde{\mathfrak{J}}^\b)''$ is also non-negative, so $(\tilde{\mathfrak{J}}^\b)'$ is strictly positive for any $s\geq s_0$. This is a contradiction to $\lim_{s\rightarrow\infty}(\tilde{\mathfrak{J}}^\b)'(s)=0$ which follows from that $\tilde{\mathfrak{J}}^\b$ is bounded and monotonic.
\end{proof}
\begin{rem}
When $\chi$ is strictly negative, using \lemref{2nd J functional} again, we see that $\frac{1}{V}\int_M \chi_{i\bar j}\dot\gamma^i\dot\gamma^{\bar j}\om_\gamma^n=0$. This implies $\gamma(s)$ is just a point which coincides with $\vphi_0$.  Therefore $\rho(t)$ will converges to $\vphi_0$.
\end{rem}

\begin{rem}
If a $C^{1,1}$ geodesic ray $\gamma(t)$ is destabilizing, then  the ${\tilde{\mathfrak{J}}}^\b$-functional is non-increasing when $t$ is large enough.
\end{rem}

\section{Proof of \thmref{geodesicstability}}
\begin{proof}
Due to \thmref{c1 nega j low bdd}, $\vphi_0$ is a global minimiser. So $\tilde{\mathfrak{J}}^\b$ is non-decreasing along any $C^{1,1}$ geodesic ray $\rho(t)$. So the first statement holds.
For the second statement, we consider the sign of $\mathfrak{F}^\b$ and prove by contradiction method. Assume that $\rho(t)$ is a geodesic ray with uniform $C^{1,1}$ bund and $\mathfrak{F}^\b$ is strictly negative along it. So according to the definition of $\mathfrak{F}^\b$ \eqref{euro}, when $t$ is large enough,
$$d\tilde{\mathfrak{J}}^\b(\frac{\p\rho}{\p t})_{\rho(t)}<0.$$ According to \propref{c1 nega geo conv}, $\rho(t)$ will converges to a $\tilde{\mathfrak{J}}^\b$-metric and $\mathfrak{F}^\b=0$. Contradiction! So the theorem follows.
\end{proof}
\section{Proof of \thmref{proper}}
Recall the entropy $$E_\om(\vphi)=\frac{1}{V}\int_M\log\frac{\om_\vphi^n}{\om^n}\om_\vphi^n.$$
The proof of \thmref{proper} follows from the following lemma and \eqref{decomposek}.

\begin{lem}(Tian \cite{MR1787650})\label{tianinequality}
There is a uniform constant $C=C(\om)>0$,
\begin{align}\label{lowered}
E_\om(\vphi)\geq \a
I_{\om}(\varphi) -C, \forall \vphi\in\mathcal H.
\end{align}
\end{lem}
\begin{proof}
The $\a$-invariant was introduced by Tian \cite{MR894378}:
\begin{align*}
\a([\om])=\sup\{\a>0\vert\exists C>0, & \text{ s.t. }
\int_M e^{-\a(\vphi-sup_M\vphi)}\om^n\leq C \\
&\text{ holds for all }\vphi\in \mathcal H\}>0.
\end{align*}
From the definition of the $\a$-invariant
\begin{align*}
\int_M e^{-\a(\vphi-\frac{1}{V}\int_M\vphi\om^n)-h}\om_\vphi^n
&=\int_M e^{-\a(\vphi-\frac{1}{V}\int_M\vphi\om^n)}\om^n\\
&\leq\int_M e^{-\a(\vphi-sup_M\vphi)}\om^n
\end{align*}
and then the Jensen inequality
\begin{align*}
\int_M [\a(-\vphi+\frac{1}{V}\int_M\vphi\om^n)-\log\frac{\om_\vphi^n}{\om^n}]\om_\vphi^n
\leq C,
\end{align*}
we obtain the lower bound of the entropy.
\end{proof}

\begin{lem}\label{Ipropertokproper}
$I$-properness of Ding functional implies $I$-properness of Mabuchi $K$-energy.
\end{lem}
\begin{proof}
From assumption, in $\Om=C_1(M)$, there are two positive constants $A_3$ and $A_4$ such that for all $\vphi\in \mathcal H_\Om$,
\begin{align}\label{Fproper}
F_\om(\vphi) \geq A_3 I_\om(\vphi)-A_4.
\end{align}
Let $f$ be the scalar potential which is defined to be the solution of the equation
\begin{equation*}
\tri_\vphi f=S-\ul{S}
\end{equation*}
with the normalisation condition
\begin{equation*}
\int_{M}e^{f}\om^n_{\vphi}=V.
\end{equation*}
Ding-Tian \cite{dingtianf} introduced the following energy functional
\begin{align*}
A(\vphi)=\frac{1}{V}\int_Mf\om_\vphi^n.
\end{align*}
Let $\mathcal H_0$ be the space of K\"ahler potential $\vphi$ under the normalization condition
\begin{align*}
\int_Me^{-\vphi+h_\om}\om^n=V.
\end{align*}
In $\mathcal H_0$, the relation between Mabuchi $K$-energy and Ding $F$-functional is
\begin{align*}
F_\om(\vphi)=\nu_\om(\vphi)+A(\vphi)-A(0).
\end{align*}
Applying the Jensen inequality to the normalization condition of $f$, we have $A(\vphi)\leq 0$. Thus the $I$-properness of Mabuchi $K$-energy is achieved by another positive constant $A_5$ from \eqref{Fproper},
\begin{align*}
\nu_\om(\vphi) \geq A_3 I_\om(\vphi)-A_5.
\end{align*}
\end{proof}
\section{Proof of \thmref{constgeodesicray}}
We construct the required geodesic ray by using the $\tilde{\mathfrak J}^\b$-flow.
\begin{prop}
Assume that the $\tilde{\mathfrak J}^\b$-flow converges to a $\tilde{\mathfrak J}^\b$-metric. From any K\"ahler potential $\psi$, there exists a semi-destabilising $C^{1,1}$-geodesic ray such that
\begin{enumerate}
\item $\tilde{\mathfrak J}^\b$ is bounded from below,
\item the infimum of the energy $\tilde E^\b$ is zero.
\end{enumerate}
\end{prop}
\begin{proof}
We connect $\psi$ to the $\tilde{\mathfrak J}^\b$-flow $\vphi(t)$ with the $C^{1,1}$-geodesic $\vphi_t(s)$. Then we define $\rho(s)=\lim_{t\rightarrow \infty}\vphi_t(s)$. Since the $\tilde{\mathfrak J}^\b$-flow $\vphi(t)$ satisfies two conclusions in this proposition and the end-points of each $\rho_t(s)$ are all in $\vphi(t)$, so $\rho(s)$ also satisfies these two conclusion automatically. The semi-destabilising is proved as following.
\begin{align*}
\mathfrak{F}^\b(\rho)&=\lim_{s\rightarrow \infty}\delta\tilde {\mathfrak J}^\b(\frac{\p\rho}{\p s})_{\rho(s)}\\
&\leq \lim_{s\rightarrow \infty} \lim_{t\rightarrow \infty} d\tilde {\mathfrak J}^\b(\frac{\p\rho}{\p s}-\frac{\p\vphi}{\p t})_{\rho_t(s)}+d\tilde {\mathfrak J}^\b(\frac{\p\vphi}{\p t})_{\rho_t(s)}\\
&=\lim_{s\rightarrow \infty} \lim_{t\rightarrow \infty} d\tilde {\mathfrak J}^\b(\frac{\p\rho}{\p s}-\frac{\p\vphi}{\p t})_{\vphi(t)}+d\tilde {\mathfrak J}^\b(\frac{\p\vphi}{\p t})_{\vphi(t)}\\
&\leq\lim_{s\rightarrow \infty} \lim_{t\rightarrow \infty} d\tilde {\mathfrak J}^\b(\frac{\p\rho}{\p s}-\frac{\p\vphi}{\p t})_{\vphi(t)}.
\end{align*}
From \eqref{lowerbound tri}, we further have the right hand side is bounded by
\begin{align*}
&\leq\lim_{s\rightarrow \infty} \lim_{t\rightarrow \infty} \sqrt{\tilde E^\b(\vphi(t))}\sqrt{2-2(\frac{\p\rho}{\p s},\frac{\p\vphi}{\p t})_{\vphi(t)}}\\
&\leq\lim_{s\rightarrow \infty} \lim_{t\rightarrow \infty} \sqrt{\tilde E^\b(\vphi(t))}\frac{C\cdot d(\vphi_0,\rho(0))}{t}\\
&=0.
\end{align*}
Thus, the proposition holds.
\end{proof}

Now we prove the convergence of the negative gradient flow $\tilde {\mathfrak J}^\b$-functional.
Assume that there is a $\om\in\Om$ such that
\begin{align}\label{jflowcondition}
(-nc_\b\cdot\om+(n-1)\chi)\wedge\om^{n-2}>0.
\end{align}
and
\begin{align}\label{convex}
-\chi>0.
\end{align}
\begin{prop}\label{conjflow}
The conditions \eqref{jflowcondition} and \eqref{convex} is equivalent to convergence of the $\tilde{\mathfrak J}^\b$-flow to a $\tilde{\mathfrak J}^\b$-metric.
\end{prop}
The shot tome existence from the fact that the linearisation operator $L$ is elliptic. In the following, we prove the a priori estimates. As long as we have the second order estimate and the zero estimate, the $C^{2,\a}$ estimate follows from the Evans-Krylov estimate. The higher order estimates is obtained by the bootstrap method.

Recall the $\tilde{\mathfrak J}^\b$-flow,
 \begin{align}\label{jflow}
\dot\vphi&=-c_\beta+\frac{n\chi\wedge\om_\vphi^{n-1}}{\om_\vphi^n}
 -\frac{\b}{V}\frac{\om^n}{\om^n_\vphi}.
 \end{align}
 We take derivative $\p_t$ on the both sides,
 \begin{align}\label{linearjflow}
\ddot\vphi
&=-\dot\vphi^{i\bar j}\chi_{i\bar j}+\frac{\b}{V}\tri_\vphi\dot\vphi \frac{\om^n}{\om^n_\vphi}\\
&=\dot\vphi_{i\bar j}[-g_\vphi^{k\bar j}g_\vphi^{i\bar l}\chi_{k\bar l}+\frac{\b}{V} \frac{\om^n}{\om^n_\vphi}g_\vphi^{ i\bar j}]\nonumber.
 \end{align}
 We denote
\begin{align}\label{Loperator}
L=[-g_\vphi^{k\bar j}g_\vphi^{i\bar l}\chi_{k\bar l}+\frac{\b}{V} \frac{\om^n}{\om^n_\vphi}g_\vphi^{ i\bar j}]\p_i\p_{\bar j}.
 \end{align}
From \eqref{convex}, we see that on the short time interval, $L$ is an elliptic operator, i.e.
\begin{align}\label{convex st}
-\chi+\frac{\b}{V}\frac{\om^n}{\om^n_\vphi} \om_\vphi>0.
\end{align} 
From the maximum principle, we have
 \begin{align}\label{dotvphi}
 \min_M \dot\vphi(0)\leq \dot\vphi(t)\leq  \max_M \dot\vphi(0).
\end{align}

\subsection{Lower bound of the 2nd derivatives}\label{2ndlower}
Using the flow equation, we have
 \begin{align*}
 \min_M \dot\vphi(0)\leq \dot\vphi(t)&=-c_\beta+\frac{n\chi\wedge\om_\vphi^{n-1}}{\om_\vphi^n}
 -\frac{\b}{V}\frac{\om^n}{\om^n_\vphi}\\
 &=-c_\beta+g_\vphi^{i\bar j}\chi_{i\bar j}
 -\frac{\b}{V}\frac{\om^n}{\om^n_\vphi}\\
 &\leq -c_\beta+g_\vphi^{i\bar j}\chi_{i\bar j}.
 \end{align*}

In the following, we always use the normal coordinate diagonalize $\om$ and $\om_\vphi$ such that their eigenvalues are $1$ and $\l_i$ for $1\leq i\leq n$ respectively. Denote the diagonal of $\chi$ by $\mu_i$.

Thus for any $1\leq i\leq n$,
 \begin{align*}
\frac{-\mu_i}{\l_i}\leq  \min_M \dot\vphi(0)-c_\beta,
 \end{align*}
 or
  \begin{align*}
\l_i\geq \frac{-\mu_i}{\min_M \dot\vphi(0)-c_\beta} .
 \end{align*}
\subsection{Upper bound of the 2nd derivatives}
Let $$A=\chi^{i\bar j}g_{\vphi i\bar j}.$$
When we work on the second order estimate, the extra term in the equation cause the trouble, we overcome it by using the linearisation operator $L$ as the elliptic operator.
Then we compute $$(\p_t-L )(\log A-C\vphi).$$
Let $$B=g_\vphi^{p\bar q}\chi_{p\bar q}.$$
We have
\begin{align}\label{bij}
B_{i\bar j}&=[g_\vphi^{p\bar q}\chi_{p\bar q}]_{i\bar j}
=-(g^{r\bar q}_\vphi g^{p\bar s}_\vphi (g_{\vphi r\bar s})_i)_{\bar j}\chi_{p\bar q}
-g^{p\bar q}_\vphi R_{p\bar q i\bar j}(\chi)\\
&=[-g^{r\bar q}_\vphi g^{p\bar s}_\vphi (g_{\vphi r\bar s})_{i\bar j}
+g^{r\bar q}_\vphi g^{p\bar b}_\vphi g^{a\bar s}_\vphi (g_{\vphi a\bar b})_{\bar j}(g_{\vphi r\bar s})_{i}\nonumber\\
&+g^{r\bar b}_\vphi g^{a\bar q}_\vphi g^{p\bar s}_\vphi(g_{\vphi a\bar b})_{\bar j}(g_{\vphi r\bar s})_i]\chi_{p\bar q}
-g^{p\bar q}_\vphi R_{p\bar q i\bar j}(\chi)\nonumber.
\end{align}
So using the flow equation,
\begin{align}\label{pta}
\p_t A&= \chi^{i\bar j}\dot\vphi_{i\bar j}\\
&=\chi^{i\bar j}[-c_\beta+g_\vphi^{p\bar q}\chi_{p\bar q}
 -\frac{\b}{V}\frac{\om^n}{\om^n_\vphi}]_{i\bar j}\nonumber\\
& =\chi^{i\bar j}  [-g^{r\bar q}_\vphi g^{p\bar s}_\vphi (g_{\vphi r\bar s})_{i\bar j}
+g^{r\bar q}_\vphi g^{p\bar b}_\vphi g^{a\bar s}_\vphi (g_{\vphi a\bar b})_{\bar j}(g_{\vphi r\bar s})_{i}\nonumber\\
&
+g^{r\bar b}_\vphi g^{a\bar q}_\vphi g^{p\bar s}_\vphi(g_{\vphi a\bar b})_{\bar j}(g_{\vphi r\bar s})_i]\chi_{p\bar q}
-g^{p\bar q}_\vphi R_{p\bar q }(\chi)
 -\frac{\b}{V}(\frac{\om^n}{\om^n_\vphi})_{i\bar j}\chi^{i\bar j}\nonumber.
\end{align}
Then computing under normal coordinate of $\om$,
\begin{align}\label{volij}
(\frac{\om^n}{\om^n_\vphi})_{i\bar j}
&=[g^{k\bar l}(g_{k\bar l})_i\om^n(\om^n_\vphi)^{-1}-\om^n (\om^n_\vphi)^{-1} g_\vphi^{k\bar l}(g_{\vphi k\bar l})_i]_{\bar j}\\
&=-g^{k\bar l}R_{k\bar l i \bar j}(\om)\om^n(\om^n_\vphi)^{-1}
+\om^n(\om^n_\vphi)^{-1} g_\vphi^{p\bar q}(g_{\vphi p\bar q})_{\bar j}g_\vphi^{k\bar l}(g_{\vphi k\bar l})_i\nonumber\\
&+\om^n (\om^n_\vphi)^{-1} g_\vphi^{k\bar q}g_\vphi^{p\bar l}(g_{\vphi p\bar q})_{\bar j}(g_{\vphi k\bar l})_i
-\om^n (\om^n_\vphi)^{-1} g_\vphi^{k\bar l}(g_{\vphi k\bar l})_{i\bar j}\nonumber.
\end{align}
Again,
\begin{align}\label{akl}
 A_{k\bar l}
&= [\chi^{p\bar q}g_{\vphi p\bar q}]_{k\bar l}\\
&=  {R^{p\bar q}}_{k\bar l}(\chi) g_{\vphi p\bar q}+\chi^{p\bar q}(g_{\vphi p\bar q})_{k\bar l}\nonumber.
\end{align}
Furthermore, from the flow equation,
\begin{align}\label{ptl}
&(\p_t-L)\vphi\\
&=-c_\beta+g_\vphi^{i\bar j}\chi_{i\bar j}
 -\frac{\b}{V}\frac{\om^n}{\om^n_\vphi}+[g_\vphi^{k\bar j}g_\vphi^{i\bar l}\chi_{i\bar j}-\frac{\b}{V} \frac{\om^n}{\om^n_\vphi}g_\vphi^{ k\bar l}] \vphi_{k\bar l}\nonumber\\
&=-c_\beta+2g_\vphi^{i\bar j}\chi_{i\bar j}
-g_\vphi^{k\bar j}g_\vphi^{i\bar l}\chi_{i\bar j}g_{k\bar l}
 -\frac{\b}{V}\frac{\om^n}{\om^n_\vphi}(n+1) + \frac{\b}{V}\frac{\om^n}{\om^n_\vphi}g_\vphi^{ k\bar l}g_{k\bar l}.\nonumber
\end{align}
Putting them together, we obtain
\begin{align}\label{ptLc}
&(\p_t-L)[\log A-C\vphi]\\
&=\frac{1}{A}\p_t A + g_\vphi^{k\bar j}g_\vphi^{i\bar l}\chi_{i\bar j} (\frac{A_{k\bar l}}{A}-\frac{A_kA_{\bar l}}{A^2})
-\frac{\b}{V}\frac{\om^n}{\om^n_\vphi}g_\vphi^{ k\bar l}
(\frac{A_{k\bar l}}{A}-\frac{A_kA_{\bar l}}{A^2})\nonumber\\
&-C[(\p_t-L)\vphi]\nonumber\\
&=\frac{\p_t A+g_\vphi^{k\bar j}g_\vphi^{i\bar l}\chi_{i\bar j}A_{k\bar l}}{A} -\frac{g_\vphi^{k\bar j}g_\vphi^{i\bar l}\chi_{i\bar j}A_kA_{\bar l}}{A^2}\nonumber\\
&-\frac{\b}{V}\frac{\om^n}{\om^n_\vphi}
\frac{g_\vphi^{ k\bar l}A_{k\bar l}}{A}+\frac{\b}{V}\frac{\om^n}{\om^n_\vphi}\frac{g_\vphi^{ k\bar l}A_kA_{\bar l}}{A^2}\nonumber\\
&-C[-c_\beta+2g_\vphi^{i\bar j}\chi_{i\bar j}
-g_\vphi^{k\bar j}g_\vphi^{i\bar l}\chi_{i\bar j}g_{k\bar l}
 -\frac{\b}{V}\frac{\om^n}{\om^n_\vphi}(n+1) + \frac{\b}{V}\frac{\om^n}{\om^n_\vphi}g_\vphi^{ k\bar l}g_{k\bar l}]\nonumber.
\end{align}
The first line in the last identity is,
\begin{align}
&\frac{\p_t A+g_\vphi^{k\bar j}g_\vphi^{i\bar l}\chi_{i\bar j}A_{k\bar l}}{A} -\frac{g_\vphi^{k\bar j}g_\vphi^{i\bar l}\chi_{i\bar j}A_kA_{\bar l}}{A^2}\nonumber\\
&= \frac{1}{A}[ -\chi^{i\bar j} g^{r\bar q}_\vphi g^{p\bar s}_\vphi (g_{\vphi r\bar s})_{i\bar j}\chi_{p\bar q}
+\chi^{i\bar j} g^{r\bar q}_\vphi g^{p\bar b}_\vphi g^{a\bar s}_\vphi (g_{\vphi a\bar b})_{\bar j}(g_{\vphi r\bar s})_{i}\chi_{p\bar q}\nonumber\\
&+\chi^{i\bar j} g^{r\bar b}_\vphi g^{a\bar q}_\vphi g^{p\bar s}_\vphi(g_{\vphi a\bar b})_{\bar j}(g_{\vphi r\bar s})_i\chi_{p\bar q}
-g^{p\bar q}_\vphi R_{p\bar q }(\chi)
 -\frac{\b}{V}(\frac{\om^n}{\om^n_\vphi})_{i\bar j}\chi^{i\bar j}]\nonumber\\
&+\frac{1}{A} g_\vphi^{k\bar j}g_\vphi^{i\bar l}\chi_{i\bar j}[ {R^{p\bar q}}_{k\bar l}(\chi) g_{\vphi p\bar q}+\chi^{p\bar q}(g_{\vphi p\bar q})_{k\bar l}] -\frac{g_\vphi^{k\bar j}g_\vphi^{i\bar l}\chi_{i\bar j}A_kA_{\bar l}}{A^2}\nonumber\\
&= \frac{1}{A}\{\chi^{i\bar j} g^{r\bar q}_\vphi g^{p\bar b}_\vphi g^{a\bar s}_\vphi (g_{\vphi a\bar b})_{\bar j}(g_{\vphi r\bar s})_{i}\chi_{p\bar q}
+\chi^{i\bar j} g^{r\bar b}_\vphi g^{a\bar q}_\vphi g^{p\bar s}_\vphi(g_{\vphi a\bar b})_{\bar j}(g_{\vphi r\bar s})_i\chi_{p\bar q}\label{3order0}\\
&-g^{p\bar q}_\vphi R_{p\bar q }(\chi) +\frac{\b}{V}\chi^{i\bar j}g^{k\bar l}R_{k\bar l i \bar j}(\om)\frac{\om^n}{\om^n_\vphi}\nonumber\\
&
-\frac{\b}{V}\chi^{i\bar j}\frac{\om^n}{\om^n_\vphi} g_\vphi^{p\bar q}(g_{\vphi p\bar q})_{\bar j}g_\vphi^{k\bar l}(g_{\vphi k\bar l})_i
-\frac{\b}{V}\chi^{i\bar j}\frac{\om^n}{\om^n_\vphi} g_\vphi^{k\bar q}g_\vphi^{p\bar l}(g_{\vphi p\bar q})_{\bar j}(g_{\vphi k\bar l})_i\label{3order1}\\
&
+\frac{\b}{V}\chi^{i\bar j}\frac{\om^n}{\om^n_\vphi} g_\vphi^{k\bar l}(g_{\vphi k\bar l})_{i\bar j}\}
+\frac{1}{A} g_\vphi^{k\bar j}g_\vphi^{i\bar l}\chi_{i\bar j}{R^{p\bar q}}_{k\bar l}(\chi) g_{\vphi p\bar q}\nonumber\\
&-\frac{g_\vphi^{k\bar j}g_\vphi^{i\bar l}\chi_{i\bar j}A_kA_{\bar l}}{A^2}\label{3order3}.
\end{align}
Here we use the identity to cancel the first term in the 2rd line and the second term in the 4th line,
\begin{align*}
(g_{\vphi p\bar q})_{k\bar l}=R_{p\bar qk\bar l}+\frac{\p^4}{\p z^{p}\p z^{\bar q}\p z^{k} \p z^{\bar l}}\vphi=R_{k\bar lp\bar q}+\frac{\p^4}{\p z^{p}\p z^{\bar q}\p z^{k} \p z^{\bar l}}\vphi=(g_{\vphi k\bar l})_{p\bar q}.
\end{align*}
The second line in the last identity in \eqref{ptLc} is
\begin{align*}
-\frac{\b}{V}\frac{\om^n}{\om^n_\vphi}
\frac{g_\vphi^{ k\bar l}[ {R^{p\bar q}}_{k\bar l}(\chi) g_{\vphi p\bar q}+\chi^{p\bar q}(g_{\vphi p\bar q})_{k\bar l}]}{A}+\frac{\b}{V}\frac{\om^n}{\om^n_\vphi}\frac{g_\vphi^{ k\bar l}A_kA_{\bar l}}{A^2}.
\end{align*}
In order to annihilate the 2nd term with 2nd term in \eqref{3order1} and 2nd term in \eqref{3order0} with \eqref{3order3}, we need the lemma,
\begin{lem}The following lemma holds.
\begin{align*}
&[\chi^{i\bar j} g_\vphi^{k\bar q}g_\vphi^{p\bar l}(g_{\vphi p\bar q})_{\bar j}(g_{\vphi k\bar l})_i]A\geq g_\vphi^{ k\bar l}A_kA_{\bar l},\\
&[\chi^{i\bar j} g^{r\bar b}_\vphi g^{a\bar q}_\vphi g^{p\bar s}_\vphi(g_{\vphi a\bar b})_{\bar j}(g_{\vphi r\bar s})_i\chi_{p\bar q}]A
\geq g_\vphi^{k\bar j}g_\vphi^{i\bar l}\chi_{i\bar j}A_kA_{\bar l}.\end{align*}
\end{lem}
\begin{proof}
Under the normal chordate of $\chi$ which is negative-defined, and $\om_\chi$ is diagonalized, the first inequality becomes,
\begin{align*}
&[g_\vphi^{k\bar q}g_\vphi^{p\bar l}\sum_{i}(g_{\vphi p\bar q})_{i}(g_{\vphi k\bar l})_i]\sum_{i}g_{\vphi i\bar i}\geq g_\vphi^{ k\bar k}(\sum_{i}g_{\vphi i\bar i})_k(\sum_{i}g_{\vphi i\bar i})_{\bar k}.
\end{align*}
This follows from the H\"older's inequality.
The second inequality is proved in Lemma 3.2 in \cite{MR2104082}.
\end{proof}
Thus \eqref{ptLc} becomes
\begin{align}\label{ptLcs}
&(\p_t-L)[\log A-C\vphi]\\
&= \frac{1}{A}\{-g^{p\bar q}_\vphi R_{p\bar q }(\chi) +\frac{\b}{V}\chi^{i\bar j}g^{k\bar l}R_{k\bar l i \bar j}(\om)\frac{\om^n}{\om^n_\vphi}\}\nonumber\\
&
+\frac{1}{A} g_\vphi^{k\bar j}g_\vphi^{i\bar l}\chi_{i\bar j}{R^{p\bar q}}_{k\bar l}(\chi) g_{\vphi p\bar q}-\frac{\b}{V}\frac{\om^n}{\om^n_\vphi}
\frac{g_\vphi^{ k\bar l} {R^{p\bar q}}_{k\bar l}(\chi) g_{\vphi p\bar q}}{A}\nonumber\\
&-C[-c_\beta+2g_\vphi^{i\bar j}\chi_{i\bar j}
-g_\vphi^{k\bar j}g_\vphi^{i\bar l}\chi_{i\bar j}g_{k\bar l}
 -\frac{\b}{V}\frac{\om^n}{\om^n_\vphi}(n+1) + \frac{\b}{V}\frac{\om^n}{\om^n_\vphi}g_\vphi^{ k\bar l}g_{k\bar l}].\nonumber
\end{align}
Since $\om_\vphi$ has lower bound from Subsection \ref{2ndlower}, the first four terms and the 4th term in the last line are bounded above by constant $C_1$, thus at the maximum point $p$ of $\log A-C\vphi$,
\begin{align*}
0\leq C_1 - C [-c_\beta+2g_\vphi^{i\bar j}\chi_{i\bar j}
-g_\vphi^{k\bar j}g_\vphi^{i\bar l}\chi_{i\bar j}g_{k\bar l}].
\end{align*}
Written in the normal co-ordinate where $\chi$ has negative diagonal $\mu_i$, it becomes
\begin{align}\label{mainineq}
0\leq C_1-C[-c_\b+2\sum_{i=1}^n\frac{\mu_i}{\l_i} - \sum_{i=1}^n\frac{\mu_i}{\l_i^2}].
\end{align}
From the condition,
 \begin{align*}
(-nc_\b\cdot\om+(n-1)\chi)\wedge\om^{n-2}>0,
\end{align*}
We have there is positive constant $\delta$ such that
 \begin{align*}
(-nc_\b\cdot\om+(n-1)\chi)\wedge\om^{n-2}\geq\delta\om^{n-1},
\end{align*}
then
 \begin{align*}
-c_\b+\sum_{i=1,i \neq k}^n\mu_i\geq \delta.
\end{align*}
From \eqref{mainineq}, we have for large $C$,
\begin{align*}
-c_\b+2\sum_{i=1}^n\frac{\mu_i}{\l_i} - \sum_{i=1}^n\frac{\mu_i}{\l_i^2} \leq \frac{C_1}{C}\leq 0.5\delta.
\end{align*}
We choose $1\leq k\leq n$ and consider,
\begin{align*}
0&\geq \sum_{i=1,i\neq k}^n\mu_i(\frac{1}{\l_i}-1)^2+\frac{\mu_k}{\l_k^2}\\
&=c_\b-2\sum_{i=1}^n\frac{\mu_i}{\l_i} + \sum_{i=1}^n\frac{\mu_i}{\l_i^2}-[c_\b-\sum_{i=1,i\neq k}^n\mu_i-2\frac{\mu_k}{\l_k}]\\
&\geq -0.5\delta+\delta+2\frac{\mu_k}{\l_k}.
\end{align*}
Thus,
\begin{align*}
\l_k\leq \frac{-4\mu_k}{\delta},
\end{align*}
or at $p$,
\begin{align*}
\om_\vphi\leq \frac{-4}{\delta}\chi.
\end{align*}
Therefore, we obtain that at any $x\in M$
\begin{align*}
\log A(x)-C\vphi(x)\leq \log A(p)-C\vphi(p),
\end{align*}
then,
\begin{align*}
\log A(x)\leq \log \frac{4n}{\delta}-C\cdot(\vphi-\inf\vphi).
\end{align*}
Therefore, there is constant $C$ such that
\begin{align}\label{2ndupper}
\om_\vphi\leq e^{C_1\cdot(\vphi-\inf\vphi)}.
\end{align}
\subsection{Zero order estimate}
It suffices to obtain the iteration formula.
Letting $$C_2=\max\{1,-\dot\vphi-c_\beta+1\}$$ from \eqref{jflow}, we have
 \begin{align*}
\om_\vphi^n\leq(\dot\vphi+c_\beta+C_2)\om_\vphi^n=n\om_\vphi^{n-1}\wedge\chi-\frac{\b}{V}\om^n+C_2\om_\vphi^n.
 \end{align*}
We compute that
 \begin{align}\label{zeroindentity}
&\om_\vphi^n-\om_\vphi^{n-1}\wedge\om\\
&\leq(\dot\vphi+c_\beta+C_2)\om_\vphi^n-\om_\vphi^{n-1}\wedge\om\nonumber\\
&=n\om_\vphi^{n-1}\wedge\chi-\frac{\b}{V}\om^n+C_2\om_\vphi^n-\om_\vphi^{n-1}\wedge\om\nonumber.
 \end{align}
 Then we let $\phi=\vphi-\inf\vphi$ and $u=e^{-C_3\phi}$, we multiply \eqref{zeroindentity} with $u$ and itegrate over $M$.
The right hand side becomes,
 \begin{align*}
&\int_M u [\om_\vphi^n-\om_\vphi^{n-1}\wedge\om]\\
&=\int_M e^{-C_3\phi} [\om_\vphi^n-\om_\vphi^{n-1}\wedge\om]\\
&=C_3\int_M e^{-C_3\phi} \p\vphi\wedge\bar\p\vphi\wedge\om_\vphi^{n-1}\\
&=C_3\int_M e^{-\frac{C_3}{2}\phi} \p\vphi\wedge e^{-\frac{C_3}{2}\phi} \bar\p\vphi\wedge\om_\vphi^{n-1}\\
&=\frac{4}{C_3}\int_M  \p u^{\frac{1}{2}}\wedge \bar\p u^{\frac{1}{2}}\wedge\om_\vphi^{n-1}\\
&\geq\frac{C_4}{C_3}\int_M | \p u^{\frac{1}{2}}|_\om^2\om^{n}.
 \end{align*}
In the last inequality we used the lower bound of $\om_\vphi$.
 While, the right hand side is
  \begin{align*}
&\int_M u [n\om_\vphi^{n-1}\wedge\chi-\frac{\b}{V}\om^n+C_2\om_\vphi^n-\om_\vphi^{n-1}\wedge\om]\\
&\leq C_2\int_M u \om_\vphi^n\\
&\leq C_2\int_M e^{-C_3\phi}  e^{C_1\cdot(\vphi-\inf\vphi)} \om^n\\
&\leq C_2\int_M e^{-C_3\phi}  e^{C_1\cdot\phi}e^{-C_1\cdot\inf\phi} \om^n\\
&\leq C_2 ||u||_0^\frac{C_1}{C_3} \int_M e^{C_3(-1+\frac{C_1}{C_3})\cdot\phi} \om^n.
 \end{align*}
We apply \eqref{2ndupper} in the second inequality.
Let $v=e^{-C_5\phi}$. We choose $C_3=p C_5$ and $\frac{C_1}{C_5}=1-\delta$, we thus obtain
  \begin{align*}
\int_M | \p v^{\frac{p}{2}}|_\om^2\om^{n}&\leq p C_6 ||v||_0^{1-\delta}\int_M  e^{C_5(-p+1-\delta)\phi} \om^n\\
&\leq p C_6 ||v||_0^{1-\delta}\int_M  v^{C_5(p-1+\delta)} \om^n.
 \end{align*}
Thus the zero order estimate follows from the iteration Lemma 3.3 in \cite{MR2226957}.




\bibliography{bib}
\bibliographystyle{plain}
\end{document}